\documentclass[11pt, a4paper]{amsart}

\usepackage[foot]{amsaddr}

\usepackage{epsfig,amsfonts,amssymb, url, amsthm, subcaption, hyperref, tikz,
amsmath, calrsfs, bigdelim, multirow, appendix, float, mathtools,
stmaryrd, environ, enumitem, dsfont}

\usepackage[utf8]{inputenc} \usepackage[T1]{fontenc}
\usepackage[top=3cm, bottom=3cm, left=2.5cm, right=2.5cm]{geometry}

\newtheorem{theorem}{Theorem}[section]
\newtheorem{lemma}[theorem]{Lemma}

\newtheorem{corollary}[theorem]{Corollary}
\newtheorem{definition}[theorem]{Definition}
\newtheorem{assumption}[theorem]{Assumption}
\newtheorem{remark}[theorem]{Remark}

\theoremstyle{definition}

\newcommand{\RR}{\mathbb{R}}
\newcommand{\NN}{\mathbb{N}}
\newcommand{\ZZ}{\mathbb{Z}}

\newcommand{\EE}{\mathbb{E}}

\newcommand{\PP}{\mathbb{P}}
\newcommand{\TT}{\mathbb{T}}
\newcommand{\DD}{\mathbb{D}}

\newcommand{\mL}{\mathcal{L}}

\newcommand{\mC}{\mathcal{C}}

\newcommand{\mM}{\mathcal{M}}

\newcommand{\mF}{\mathcal{F}}
\newcommand{\mD}{\mathcal{D}}

\newcommand{\mE}{\mathcal{E}}
\newcommand{\mS}{\mathcal{S}}

\newcommand{\mH}{\mathcal{H}}

\newcommand{\mf}[1]{\mathfrak{#1}}
\newcommand{\mfd}{\mathfrak{d}}

\renewcommand{\l}{\ell}

\newcommand{\para}{\varolessthan}

\newcommand{\reso}{\varodot}

\newcommand{\ve}{\varepsilon}
\newcommand{\vr}{\varrho}
\newcommand{\vt}{\vartheta}

\newcommand{\bigslant}[2]{{\raisebox{.1em}{$#1$}\left/\raisebox{-.1em}{$#2$}\right.}}
\newcommand*{\ud}{\mathrm{\,d}}

\newcommand{\znd}{\mathbb{Z}^d_n} 
\newcommand{\tnd}{\mathbb{T}^d_n}

\newcommand{\supp}{\operatorname{supp}}

\begin{document}

\title{Killed Rough Super-Brownian Motion}

\author{Tommaso Cornelis Rosati}
\address{Humboldt-Universit\"at zu Berlin}

\thanks{This paper was developed within the scope of the IRTG 1740 /
TRP 2015/50122-0, funded by the DFG / FAPESP}

\begin{abstract} 
  This note extends the results in \cite{PerkowskiRosati2019} to construct the
  rough super-Brownian motion on finite volume with Dirichlet boundary
  conditions. The backbone of this study is the convergence of discrete
  approximations of the parabolic Anderson model (PAM) on a box.
\end{abstract}

\maketitle 

\section{Introduction}

In \cite{PerkowskiRosati2019} a superprocess on infinite volume is constructed
(named rough super-Brownian motion, rSBM),
as a scaling limit of a branching random walk in a static random environment
(BRWRE).
In the quoted work, the analysis of persistence of the superprocess relies
on the existence of the same process on finite volume with Dirichlet boundary
conditions, due to the spectral properties of the Anderson Hamiltonian. The
construction of such process is the aim of the current work.

Such process is the scaling limit of the same branching particle system as in
\cite{PerkowskiRosati2019}, where
any particle is killed as soon as it leaves a box of size \(L\). Morally, this
scaling limit is simpler to treat than in the infinite volume case, since
explosions are less likely to happen. Indeed the convergence of the empirical
measure associated to the particle system is an application of the results in
\cite[Section 3]{PerkowskiRosati2019}.

On a more technical level, the construction in \cite{PerkowskiRosati2019} relies
on the tools of \cite{MartinPerkowski2017} for discrete approximations of the
parabolic Anderson model (PAM) on infinite volume. In this work we extend the
latter approach within the framework of \cite{Chouk2019} for paracontrolled
analysis with Dirichlet boundary conditions, with the aim of proving the
convergence of discrete approximations to PAM with
Dirichlet boundary conditions. That is, we study the equation:
\begin{equation}\label{eqn:pam_dirichlet_continuous} 
  \begin{aligned} \partial_t
  w(t,x) &= \Delta w(t,x) {+} \xi(x) w(t,x) {+} f(t,x), \qquad & (t,x) \in
  (0,T]\times(0,L)^d,\\ w(0,x) &= w_0(x), \qquad w(t, x) = 0, & (t, x) \in
  (0,T] \times \partial [0,L]^d,  
  \end{aligned} 
\end{equation}
where \(\xi\) is space white noise. The details are explained in the next section.

\proof[Acknowledgements.] We are very grateful to Nicolas Perkowski for the kind
help in the preparation of this note.  

\section{PAM with Dirichlet Boundary Conditions}\label{sect:PAM_dirichlet_bc}

Define \(\NN = \{1, 2, \dots\}, \NN_{0} = \NN \cup \{0\}\). Fix \(L \in \NN\)
and \(N = 2L\).  Consider $n \in \NN \cup \{ \infty\}$ ($n=\infty$ refers to the
continuous case, studied in \cite{Chouk2019}).  Write \(\znd\) for the lattice
\(\frac{1}{n} \ZZ^{d}\) (resp.  \(\mathbb{R}^{d}\) if \(n = \infty\)),
$\Lambda_{n}$ for the lattice $\frac{1}{n}( \mathbb{Z}^{d}\cap \left[ 0, L n
\right]^{d} )$ (resp. $[0, L]^d$), $\Theta_{n}$ for the lattice
$\bigslant{\frac{1}{n}( \mathbb{Z}^{d}\cap [{-}\frac{N n}{2}, \frac{N n}{2}
]^{d} )}{\sim }$ with opposite boundaries identified (resp.  $\TT^d_{N}\colon =
\bigslant{[{-} \frac{N}{2} , \frac{N}{2} ]^{d}}{\sim}$) and define the ``dual
lattice'' \(\Xi_{n}= \bigslant{ \frac{1}{N} ( \mathbb{Z}^{d}\cap [{-}
\frac{Nn}{2}, \frac{Nn}{2}]^{d} )}{\sim }\), (resp. \(\frac{1}{N} \ZZ^d\)) as
well as \(\Xi_{n}^{+}= \frac{1}{N} ( \mathbb{Z}^{d}\cap [0, Ln]^{d} )\), (resp.
$\frac{1}{N}\mathbb{N}_0^d$) and \(\partial \Xi^{+}_{n} = \{k \in \Xi^{+}_{n} \
: \ k_i = 0 \text{ for some } i \in \{1, \ldots, d\} \}\). Write
\(A_{\mf{d}}^{n} = \Xi_{n}^{+} \setminus
\partial \Xi_{n}^{+}, A_{\mf{n}}^{n}= \Xi_n^+ \). Finally, for \(p \geq 1\) and
any function \(f \colon \Theta^{n} \to \RR\), write \(\| f
  \|_{L^{p}(\Theta^{n})} = ( n^{{-} d} \sum_{x \in
\Theta^{n}}|f(x)|^p)^{\frac{1}{p} }\) (resp. the classical
  \(L^{p}([{-} \frac{N}{2}, \frac{N}{2} ]^{d} )\) norm if \(n
= \infty\)).

\subsection{The Analytic Setting}

The idea of \cite{Chouk2019} in the case $n=\infty$ is to consider suitable even
and odd extensions of functions on $\Lambda_n$ to periodic functions on
$\Theta_n$, and then to work with the usual tools from periodic paracontrolled
distributions on $\Theta_n$.  So for $u\colon\Lambda_n \to \RR $ we define
\begin{gather*} 
  \Pi_o u\colon \Theta_{n} \to \RR, \ \ \ \Pi_o u(\mathfrak{q} \circ x)= \prod
  \mathfrak{q} \cdot u(x), \qquad \Pi_e u \colon \Theta_n \to \RR, \ \ \ \Pi_e
  u(\mathfrak{q} \circ x) = u(x), 
\end{gather*}
where $x \in \Lambda_{n}, \mathfrak{q} \in \left\{{-} 1, 1\right\}^{d}$ and we
define the product \( \mathfrak{q} \circ x =(\mathfrak{q}_i
x_i)_{i = 1, \ldots, d}\) as well as \(\prod \mathfrak{q} =
  \prod_{i = 1}^d \mathfrak{q}_i\). 
  We shall work with the
  discrete periodic Fourier transform, defined for $\varphi \colon \Theta_n \to
  \RR$ by 
  \[ \mF_{\Theta_n} \varphi (k) = \frac{1}{n^d} \sum_{x \in
\Theta_n}  \varphi (x) e^{{-} 2 \pi \iota \langle x, k \rangle}, \ \ k \in
\Xi_n.\] 
As in \cite{Chouk2019} we have a periodic, a Dirichlet and a Neumann basis,
which we will denote with: \(\{\mathfrak{e}_{k}\}_{k \in \Xi_{n}},
\{\mathfrak{d}_k\}_{k \in \Xi_{n}^{+} \setminus \partial \Xi_{n}^{+}}\) and ,
\(\{\mathfrak{n}_{k}\}_{k \in \Xi_{n}^{+}}\) respectively. Here
\(\mathfrak{e}_{k}\) is the classical Fourier basis:
\[ \mathfrak{e}_k(x) =
  \frac{e^{2\pi \iota \langle x, k \rangle}}{N^{\frac{d}{2}}}, \ \ \ \text{ so
  that } \ \ \mF_{\Theta_n} \varphi (k) = N^{\frac{d}{2}}\langle \varphi,
\mathfrak{e}_k \rangle, \ \ k \in \Xi_n,
\]
the Dirichlet and Neumann bases consists of sine and cosine functions
respectively:
\[
  \mathfrak{d}_k(x) = \frac{1}{N^{\frac{d}{2} }} \prod_{i=1}^d 2 \sin(2 \pi k_i
  x_i), \ k \in A_{\mf{d}}^{n} \ \  \mathfrak{n}_k(x) = \frac{1}{N^{\frac{d}{2}
  }} \prod_{i=1}^d 2^{1 - 1_{\{k_i = 0\}}/2} \cos(2\pi k_i x_i), \ k \in
  A_{\mf{n}}^{n}.
\]
To the previous explicit expressions we will prefer the following alternative
characterization (with \(\nu_k =  2^{{-} \# \{i : k_i =0\}/2}\)): 
\begin{align*}
    \Pi_o{\mathfrak{d}}_k= \iota^{d} \sum_{\mathfrak{q} \in  \{ {-} 1,
    1\}^{d}} \prod \mathfrak{q} \cdot \mathfrak{e}_{\mathfrak{q} \circ k},
    \ \ \ \forall k \in A_{\mf{d}}^{n}, \ \ \ \ \Pi_e
    {\mathfrak{n}}_k = \nu_k \sum_{\mathfrak{q} \in \{ {-} 1, 1\}^{d}}
    \mathfrak{e}_{\mathfrak{q} \circ k}, \ \ \ \forall k \in
    A_{\mf{n}}^{n}.
\end{align*}
For $\mathfrak{l}\in \{ \mathfrak{d}, \mathfrak{n}\}$ and $n < \infty$ write \(
\mS^\prime_{\mathfrak{l}} (\Lambda_n) = \mathrm{span} \{\mathfrak{l}_k\}_{k \in
A_{\mf{l}}^{n}}\) for the space of discrete distributions. For $n = \infty$ we
define distributions via formal Fourier series: 
\[ 
  \mS^\prime_{\mathfrak{l}} ([0,L]^d) = \Big\{ \sum_{k \in A_{\mf{l}}^{\infty}}
  \alpha_k \mathfrak{l}_k \ : \ |\alpha_k| \le C(1{+}|\kappa|^\gamma), \text{
for some } C,\gamma \ge 0 \Big\}.
\]  
Now let us introduce Littlewood-Paley theory on the lattice, in order to
control products between distributions on \(\Lambda_{n}\) uniformly in $n$.
Consider an \emph{even} function \(\sigma \colon \Xi_{n}
\to \mathbb{R}.\) Then for \(\varphi \in \mS^\prime_{\mathfrak{l}} (\Lambda_n)\)
we define the \emph{Fourier multiplier}: \[ \sigma(D) \varphi = \sum_{k \in
  A_{\mf{l}}^{n}} \sigma(k) \langle
\varphi, \mathfrak{l}_{k} \rangle \mathfrak{l}_k. \]
Upon extending \(\varphi\) in an even or odd fashion we recover the classical
notion of Fourier multiplier (namely on a torus: \(\sigma(D) \varphi =
\mF_{\Theta_n}^{{-}1} (\sigma \mF_{\Theta_n} \varphi)\)), since \(
  \Pi_o\big(\sigma(D) \varphi\big) = \sigma(D) \Pi_o\varphi\) and verbatim for
  \(\Pi_{e}\).  Fix then a dyadic partition of the unity \(\{\varrho_{j} \}_{j
  \geq {-} 1} \) as in \cite[Definition 2.4]{MartinPerkowski2017} and
  let \(j_{n} = \min \{j \geq {-} 1 \colon \supp(\varrho_{j}) \not\subseteq (
  {-} \frac{N n}{2} , \frac{N n}{2} )^{d}\}\) ( \(j_{n} = \infty\) if \(n =
  \infty\)), so as to define for \(\varphi \in \mS^\prime_{\mathfrak{l}}
  (\Lambda_n)\): 
\[ 
  \Delta^{n}_{j} \varphi = \varrho_{j} (D) \varphi \text{ for } j < j_{n},
  \qquad \Delta^{n}_{j_{n}} \varphi= \Big(1 - \! \! \sum_{{-} 1 \leq j < j_{n}}
  \varrho_{j}(D) \Big) \varphi.
\]
This allows us to define the \emph{paraproduct} and the \emph{resonant product}
of two distributions respectively (for \(n = \infty\) the latter is a-piori
ill-posed):
\[
  \varphi \para \psi = \sum_{{-} 1 \leq j \leq j_{n}} \sum_{{-} 1 \leq i \leq j {-} 1}
  \Delta_{i}^{n} \varphi \Delta_{j}^{n} \psi, \qquad \varphi \reso \psi = \sum_{
  |i {-} j| \leq 1 } \Delta_{i}^{n} \varphi \Delta_{j}^{n} \psi.
\]
In view of the previous calculations this is coherent with the 
definition on the lattice in \cite{MartinPerkowski2017}, in the sense that:
\[
	\Pi_o\big(\Delta_{j}^{n}\varphi\big) = \Delta_{j}^{n} \Pi_o\varphi,
	\qquad \Pi_e\big({\Delta_{j}^{n}\varphi}\big) = \Delta_{j}^{n}
	\Pi_e{\varphi}, \ \ {-}1 \leq j \leq j_{n}.
\]
We then define Dirichlet and Neumann Besov spaces via the following norms:
\begin{align*} 
  \| u\|_{B^{\mathfrak{d}, \alpha}_{p, q}(\Lambda_n)} =
  \|\Pi_o{u}\|_{B^{\alpha}_{p, q}(\Theta_n)} = \| (2^{\alpha j } \|\Delta_j
  \Pi_o{u} \|_{L^p(\Theta_n)})_j \|_{\ell^q(\le j_n)} \qquad u \in
  \mS^{\prime}_{\mf{d}}(\Lambda_{n})
\end{align*}
and similarly for \(\mf{n}\) upon replacing \(\Pi_{o}\) with \(\Pi_{e}\). For
brevity we write $\mC^{\alpha}_{\mathfrak{l}, p}(\Lambda_n) = B^{\mathfrak{l},
\alpha}_{p,\infty}(\Lambda_n)$ and \(\mC^{\alpha}_{\mathfrak{l}}(\Lambda_n) =
B^{\mathfrak{l}, \alpha}_{\infty,\infty}(\Lambda_n)\) for $\mathfrak{l} \in \{
\mathfrak{n}, \mathfrak{d}\}$. We also write
$\|u\|_{L^p_{\mathfrak{d}}(\Lambda_n)} = \|\Pi_o u \|_{L^p(\Theta_n)}$ and
$\|u\|_{L^p_{\mathfrak{n}}(\Lambda_n)} = \|\Pi_e u \|_{L^p(\Theta_n)}$. Having
introduced Besov spaces we can define the spaces of time-dependent functions \(\mM^{\gamma}
\mC^{\alpha}_{\mf{l}, p}\) and \(\mL^{\gamma, \alpha}_{\mf{l}, \alpha}\) for
\(\mf{l} \in \{\mf{d}, \mf{n}\}\) as in
\cite[Definition 3.8]{MartinPerkowski2017} without the necessity of taking into
account weights. The above spaces allow for a detailed analysis of products of
distributions. The last ingredient in this sense are the following identities:
\begin{equation}\label{eqn:odd_even_extension_if_products} 
  \Pi_e({\varphi\psi}) = \Pi_e{\varphi} \Pi_e{\psi}, \qquad \Pi_o({\varphi\psi})
  = \Pi_o{\varphi} \Pi_e{\psi}.
\end{equation} 
To solve equations with Dirichlet boundary conditions, introduce the following
Laplace operators for $n < \infty$ (let \(\varphi\colon  \Lambda_{n} \to
\mathbb{R}\), \(\psi  \colon \Theta_{n} \to \RR\)):
\[
  \Delta^{n} \psi (x) = n^{2} \!\!\! \sum_{|x {-} y| =  n^{{-} 1}} \psi(y) {-}
  \psi(x), \qquad \Delta^{n}_{\mathfrak{d}} \varphi = (\Delta^{n} \Pi_{o}
  \varphi) \vert_{\Lambda_{n}}, \qquad \Delta_{\mf{n}}^{n} \varphi = (\Delta^{n}
  \Pi_{e} \varphi) \vert_{\Lambda_{n}}.
\] 
The latter two operators are defined only on the domain \(\mathrm{Dom}(
\Delta^{n}_{\mf{l}}) = \mS^{\prime}_{\mf{l}}(\Lambda_{n})\).
A direct computation (cf. \cite[Section~3]{MartinPerkowski2017}) then shows that
we can represent both Laplacians as Fourier multipliers:
\[ \Delta^{n}_{\mathfrak{l}} \mathfrak{l}_{k} = l^{n}(k) \mathfrak{l}_{k},
  \qquad l^{n}(k) = \sum_{j = 1}^{d}2 n^2\big(\cos{(2 \pi k_{j} /n)} {-} 1\big),
  \text{ for } \mf{l} \in \{\mf{d}, \mf{n}\}.
\] 
Note that \(l^{n}\) is an even function in \(k\), so all the remarks from the
previous discussion apply. For $n = \infty$ we use the classical Laplacian: the boundary condition is
encoded in the domain. We write $\Delta_{\mathfrak{l}}$ for the Laplacian on
$\mS^\prime_{\mathfrak{l}} ([0,L]^d)$.  
We introduce Dirichlet and Neumann extension operators as follows: \[
\mathcal{E}^{n}_{\mathfrak{d}} u = \mathcal{E}^{n} (\Pi_o{u}) \big\vert
_{[0,L]^{d}} , \qquad \mathcal{E}^{n}_{\mathfrak{n}} u = \mathcal{E}^{n}(\Pi_e{u}) \big\vert_{[0,L]^{d}}, 
\qquad \text{ for } n < \infty,\]
where the periodic extension operator \(\mE^{n}\) is defined as in
\cite[Lemma 2.24]{MartinPerkowski2017}.
These functions are well-defined since for fixed \(n\) the extension $\mathcal{E}_n(\cdot)$ is a
smooth function. Moreover a simple calculation shows that
\begin{equation}\label{eqn:dirichlet_extension_identity}
	\Pi_o(\mathcal{E}^n_{\mathfrak{d}}u) = \mathcal{E}^{n}(\Pi_o{u}), \qquad
	\Pi_e({\mathcal{E}^{n}_{\mathfrak{n}}u}) = \mathcal{E}^{n} (\Pi_e{u}).
\end{equation}

\subsection{Solving the Equation}

We now study Equation \eqref{eqn:pam_dirichlet_continuous} on a box. We start
with the crucial probabilistic assumptions on the noise (cf.
\cite[Asumption 2.1]{PerkowskiRosati2019}). 

\begin{assumption}\label{assu:noise}
  We assume that for
  every \(n \in \NN\),  $\{\xi^n(x)\}_{x \in \ZZ^d_n}$ is a set of i.i.d random
  variables which satisfy: 
  \begin{equation}\label{eqn:distr_xi} n^{-d/2}
    \xi^n(x) \sim \Phi,
  \end{equation} for a probability distribution $\Phi$
  on $\RR$ with finite moments of every order and which satisfies \[\EE[\Phi] = 0, \ \
  \EE[\Phi^2] = 1.\]
\end{assumption}

These probabilistic assumptions guarantee certain analytical properties which we
highlight in the next lemma. In the remainder of this work we shift $\Lambda_n$ to be centered around the
origin and identify it with a subset of $[-L/2,L/2]^d$. This is convenient
because later we want to interpret processes on $\Lambda_n$ as ``restrictions''
of a processes on $\znd$ to (large) boxes centered around the origin. By this we
mean that for \(L \in 2 \NN\) we define \(\Lambda_{n} = \{x \in \znd
\ \colon \ x \in [{-} L/2, L/2]^d \}\) 
In the following let \(\chi\) be the same cut-off function as in
\cite[Section 5.1]{MartinPerkowski2017} and in dimension \(d = 2\) define the
renormalization constant (note that this constant does not depend on
\(L\)):
\begin{equation}\label{eqn:renormalization_constant}
  \kappa_{n} =  \int_{\TT^{2}_{n}} \ud k \ \frac{\chi(k)}{l^{n}(k)} \sim
  log(n).
\end{equation}

\begin{lemma}\label{lem:renormalisation-neumann}
  Let \( \overline{\xi}^{n}\) be a sequence of random
  variable satisfying
  Assumption \ref{assu:noise}. There exists a probability space \( (
  \Omega, \mF, \PP)\) supporting random variables
  \(\xi^{n}, \xi\) such that \(\xi\) is space white noise on
  \(\RR^{d}\) and \(\xi^{n}= \overline{\xi}^{n}\) in distribution for every
  \(n \in \NN\).
  
  Such random variables satisfy the following requirements.
  Let $X^n_{\mf{n}}$ be the (random) solution to the equation ${-}\Delta^n_{\mathfrak{n}} X^n_{\mathfrak{n}} =
  \chi(D)\xi^n$. For every \(\omega \in
  \Omega\) and \(\alpha\) satisfying
  \begin{equation}\label{eqn:alpha_bounds_dirichlet}
	\alpha \in (1, \tfrac32) \text{ in } d = 1, \qquad \alpha \in (\tfrac23,
	1) \text{ in } d=2,
  \end{equation}
  the following holds for all \(L \in 2\NN\):
  \begin{itemize}
    \item[(i)] $\xi(\omega) \in
      \mC^{\alpha{-}2}_{\mathfrak{n}}([-L/2,L/2]^d)$ as well as \( \sup_n
      \|\xi^n(\omega)\|_{\mC_{\mathfrak{n}}^{\alpha{-}2}(\Lambda_n)}< {+}\infty\) and
      \(\mE^n_{\mathfrak{n}} \xi^n(\omega) \to \xi(\omega)\) in
      \(\mC_{\mathfrak{n}}^{\alpha{-}2}([-L/2,L/2]^d)\).

    \item[(ii)]  For any $\ve >0$ (with \( ( \cdot)_{+} = \max \{ 0, \cdot\}\)):
	\[
	  \sup_n \|n^{-d/2} \xi^n_{+}(\omega)
	  \|_{\mC_{\mathfrak{n}}^{-\ve}(\Lambda_n)} + \sup_n \|n^{-d/2}
	  |\xi^n(\omega)| \|_{\mC_{\mathfrak{n}}^{-\ve}(\Lambda_n)} + \sup_{n}
	  \| n^{ {-} d/2} \xi^n_{{+}}(\omega)
	  \|_{L^2_{\mathfrak{n}}(\Lambda_{n})} < {+}\infty.
	\]
	Moreover, there exists a \(\nu \geq 0 \) such that \(
	\mathcal{E}^n_{\mathfrak{n}} n^{{-} d/2} \xi^n_+(\omega) \to \nu,\)
	\(\mathcal{E}^n_{\mathfrak{n}}n^{{-} d/2} | \xi^n(\omega)| \to 2 \nu\) in \(
	\mathcal{C}^{ {-} \varepsilon}_{\mathfrak{n}}(\Lambda_{n})\). 

      \item[(iii)] If \(d = 2\) there exists a sequence \(c_{n}(\omega) \in
	\mathbb{R}\) such that \(n^{{-} d/2} c_{n} \to 0\) and distributions
	\(X_{\mathfrak{n}}(\omega), X_{\mathfrak{n}} \diamond \xi(\omega)\) in
	\(\mathcal{C}^{\alpha}_{\mathfrak{n}}([{-} L/2, L/2]^d)\) and
	\(\mathcal{C}^{2 \alpha {-} 2}_{\mathfrak{n}}([ {-} L/2, L/2]^d)\)
	respectively, such that: 
	\[ 
	  \sup_n \| X^n_{\mathfrak{n}}(\omega)
          \|_{\mC^{\alpha}_{\mathfrak{n}}(\Lambda_{n})} + \sup_n \|
          (X^n_{\mathfrak{n}} \reso \xi^n)(\omega) {-}c_n(\omega)
          \|_{\mC^{2\alpha{-}2}_{\mathfrak{n}}(\Lambda_{n})} <{+} \infty 
	\] 
	and $\mE^n_{\mathfrak{n}} X^n_{\mathfrak{n}}(\omega) \to
	X_{\mathfrak{n}}(\omega)$ in
	$\mC^{\alpha}_{\mathfrak{n}}([{-} L/2, L/2]^d)$,
	$\mE^n_{\mathfrak{n}}\big( (X^n_{\mathfrak{n}}  \reso\xi^n)(\omega)
	{-}c_n(\omega) \big)
	\to X_{\mathfrak{n}}\diamond \xi(\omega)$ in $\mC^{2\alpha{-}2}_{\mathfrak{n}}([
	{-} L/2, L/2]^d)$.
  \end{itemize} 
  Finally, \(\PP( c_n (\omega) = \kappa_{n}, \forall n \in \NN \text{ and }
  \nu = \EE \Phi_{+} ) = 1\) and for
  all \(\omega \in \Omega\), \(\xi^{n}(\omega)\) is a deterministic environment
  satisfying \cite[Assumption 2.3]{PerkowskiRosati2019}, with the same renormalization constant
  \(c_n(\omega)\) as above if \(d =2\).

\end{lemma}

The proof of this lemma is postponed to the next subsection. We pass to the main
analytic statement of this work.

\begin{theorem}\label{thm:convergence_PAM_Dirichlet} 
  Consider \(\xi^{n}\) as in Lemma \ref{lem:renormalisation-neumann} and \(\alpha\) as
  in \eqref{eqn:alpha_bounds_dirichlet}, any \(T >0\), $p \in [1, {+}
  \infty], \gamma_{0} \in [0, 1)$ and $\vartheta, \zeta, \alpha_0$ satisfying:
  \begin{equation}\label{eqn:condition_parameters_PAM_Dirichlet}
    \vartheta \in \begin{cases} (2{-}\alpha, \alpha), & d=1, \\ (2{-}2\alpha,
    \alpha), & d = 2, \end{cases} \qquad \zeta >(\vartheta{-}2)\vee ({-}\alpha),
    \qquad \alpha_0 >(\vartheta{-}2)\vee ({-}\alpha),
  \end{equation}
  and let $w^n_0 \in \mC^{\zeta}_{\mathfrak{d}, p}(\Lambda_n)$ and $f^n \in
  \mM^{\gamma_0}\mC^{\alpha_0}_{\mathfrak{d}, p}(\Lambda_n)$ be such that 
  \[	
    \mE^n w^n_0 \to w_0 \text{ in } \mC^{\zeta}_{\mathfrak{d}, p}([-L/2,L/2]^d),
    \qquad \mE^n f^n \to f \text{ in }
    \mM^{\gamma_0}\mC^{\alpha_0}_{\mathfrak{d}, p}([-L/2,L/2]^d).  
  \]
  Let $w^n \colon [0,T] \times \Lambda_n \to \RR$ be the unique solution to the finite-dimensional linear ODE:
  \begin{equation}\label{eqn:pam-rigorous-discrete_dirichlet} 
    \partial_t w^n = (\Delta^n_{\mathfrak{d}} + \xi^n {-} c_n 1_{\{d =2\}}) w^n
    + f^n, \ \ w^n(0) = w^n_0, \ \ w(t,x) = 0 \ \ \forall (t,x) \in (0,T]
    \times \partial \Lambda_n.
  \end{equation}
  There exist a unique (paracontrolled in the sense of \cite{Chouk2019} or
    \cite{MartinPerkowski2017} in $d=2$) solution $w$ to the equation
  \begin{equation}\label{eqn:pam-rigorous-continuous_dirichlet}
    \partial_t w = (\Delta_{\mathfrak{d}}  + \xi ) w + f, \ \ w(0)
    = w_0, \ \  w(t,x) = 0  \ \ \forall (t,x) \in (0,T] \times \partial
    [-L/2, L/2]^d,
  \end{equation}
  and for all $\gamma > (\vartheta{-}\zeta)_+ / 2 \vee \gamma_0$ the sequence
  $w^n$ is uniformly bounded in $\mL_{\mathfrak{d}, p}^{\gamma,
  \vartheta}(\Lambda_n)$:
  \begin{align*}
    \sup_{n} \| w^{n} \|_{\mL^{\gamma, \vt}_{\mf{d}, p} (\Lambda_{n})} \lesssim
    \sup_{n} \| w^{n}_{0} \|_{ \mC^{\zeta}_{\mf{d}, p}( \Lambda_{n} )} + \sup_{n}\|
    f^{n} \|_{ \mM^{\gamma_{0}} \mC^{\alpha_{0}}_{\mf{d}, p}( \Lambda_{n}) },
  \end{align*}
   where the proportionality constant depends on the time horizon \(T\) and the magnitude
   of the norms in Lemma \ref{lem:renormalisation-neumann}. Moreover,
   \[
     \mE^n w^n \to w \text{ in } \mL_{\mathfrak{d}, p}^{\gamma, \vartheta}([-L/2,L/2]^d).
   \]
\end{theorem}

\begin{proof}
  Note that in view of \eqref{eqn:odd_even_extension_if_products} solving
  Equation \eqref{eqn:pam-rigorous-discrete_dirichlet} (resp.
  \eqref{eqn:pam-rigorous-continuous_dirichlet}) is equivalent to solving on the
  discrete (resp. continuous) torus \( \Theta_{n} \) the equation:
  \[ 
    \partial_{t} \tilde{w}^n  = \Delta^n \tilde{w}^n {+} \Pi_e(\xi^n_{e})
    \tilde{w}^n {+} \Pi_o f, \qquad \tilde{w}^n(0) = \Pi_o w_0,   
  \] 
  and then restricting the solution to the cube \(\Lambda_{n}\), i.e. $w^n =
  \tilde{w}^n\vert_{\Lambda_n}$, and $\tilde{w}^n = \Pi_o w^n$. Via the bounds
  in Lemma \ref{lem:renormalisation-neumann} this equation can be solved for all
  \(\omega \in \Omega^{p}\) via Schauder estimates and (in dimension \(d =2\))
  paracontrolled theory following the arguments of \cite{MartinPerkowski2017}
  (without considering weights). From the arguments of the same article and
  Equation \eqref{eqn:dirichlet_extension_identity} we can also deduce the
  convergence of the extensions. Note that the solution theories in
  \cite{Chouk2019} and \cite{MartinPerkowski2017} coincide, although the latter
  concentrates on the construction of the Hamiltonian rather than the solutions
  to the parabolic equation (cf. \cite[Proposition 3.1]{PerkowskiRosati2019}).
\end{proof}

For every \(\omega \in \Omega\) it is also possible to define the Anderson
Hamiltonian $\mH_{\mathfrak{d}, L }^{\omega}$ with Dirichlet boundary
conditions. The domain and spectral decomposition for this operator are
rigorously constructed in \cite{Chouk2019} with the help of the resolvent
equation for \(d = 2\) and \cite{Lamarre2019}
via Dirichlet forms in \(d=1\). We write $\mH^{n,
\omega}_{\mathfrak{d}, L}, \mH_{\mathfrak{d}, L}^{\omega}$ for the operators
$\Delta^n_{\mathfrak{d}} {+} \xi^n(\omega){-}c_n(\omega)1_{\{d = 2\}}$ and
(formally) $\Delta_{\mathfrak{d}} {+} \xi(\omega){-}\infty 1_{\{d = 2\}}$
respectively. These operators generate semigroups $T^{n, \mathfrak{d}, L,
\omega}_t = e^{t\mH^{n, \omega}_{\mathfrak{d}, L}}$ and $T_t^{\mathfrak{d}, L,
\omega} = e^{t\mH_{\mathfrak{d}, L}^{\omega}}$.  In particular, the following
result is a simple consequence of the just quoted works. 

\begin{lemma}
  For a given null-set \(N_{0} \subseteq \Omega\) and all \(\omega \in N_0^{c}\),
  for all \(L \in \NN\) the operator $\mH^{\omega}_{\mathfrak d, L}$ has a
  discrete, bounded from above, spectrum and admits an eigenfunction
  $e_{\lambda(\omega, L)}$ associated to the largest eigenvalue $\lambda(\omega,
  L)$, such that $e_{\lambda(\omega, L)}(x) > 0$ for all $x \in (-\frac{L}{2},
  \frac{L}{2})^d$.
\end{lemma}

\begin{proof}
  That the spectrum is discrete and bounded from above can be found in the
  works quoted above.
  For $\varphi,\psi \in L^2((-\frac{L}{2}, \frac{L}{2})^d)$ we write $\psi \ge
  \varphi$ if $\psi(x) - \varphi(x) \ge 0$ for Lebesgue-almost all $x$ and we
  write $\psi \gg \varphi$ if $\psi(x) - \varphi(x) > 0$ for Lebesgue-almost all
  $x$. By the strong maximum principle
  of~\cite[Theorem~5.1]{Cannizzaro2017Malliavin} (which easily extends to our
  setting, see Remark~5.2 of the same paper) we know that for the semigroup
  $T^{\mathfrak d, L, \omega}_t = e^{t \mH^{\omega}_{\mathfrak d, L}}$ of the
  PAM we have $T^{\mathfrak d, L, \omega}_t \varphi \gg 0$ whenever $\varphi \ge
  0$ and $\varphi \neq 0$; we even get $T^{\mathfrak d, L, \omega}_t \varphi(x)
  > 0$ for all $x$ in the interior $(-\frac{L}{2}, \frac{L}{2})^d$. So by a
  consequence of the Krein-Rutman theorem,
  see~\cite[Theorem~19.3]{Deimling1985}, there exists an eigenfunction
  $e_{\lambda(\omega, L)} \gg 0$. And since $e_{\lambda(\omega, L)} = e^{-t
  \lambda(\omega, L)}T^{\mathfrak d, L, \omega}_t e_{\lambda(\omega, L)}$, we
  have $e_{\lambda(\omega, L)}(x) > 0$ for all $x \in (-\frac{L}{2},
  \frac{L}{2})^d$.
\end{proof}

\subsection{Stochastic Estimates}

Here we prove Lemma \ref{lem:renormalisation-neumann}. The following bounds are
essentially an adaptation of \cite[Section 4.2]{ChoukGairingPerkowski2017} to
the Dirichlet boundary condition setting (see also \cite{Chouk2019} for the
spatially continuous setting).
\begin{proof}[Proof of Lemma \ref{lem:renormalisation-neumann}] 
  \textit{ Step 0.} Let us write \(\xi^{n}\) instead of \( \overline{\xi}^{n}\).
  Fix \(L \in \NN\) and take \(\alpha, \ve\) as in the statement of the
  lemma. Instead of proving the path-wise bounds and convergences of the lemma,
  it is sufficient to prove the bounds on average and the convergences in
  distribution. By this we mean that there exists space white noise \(\xi\) on
  \(\RR^{d}\) and (if \(d = 2\)) a random distribution \(X_{\mf{n}} \diamond
  \xi\) such that (all convergences being in distribution):
  \begin{equation}\label{eqn:support_proof_lemma_eq1} 
    \sup_{n} \mathbb{E} [ \| \xi^n \|_{\mathcal{C}^{\alpha {-}
    2}_{\mathfrak{n}}(\Lambda_{n})}^q] < {+} \infty, \qquad
    \mathcal{E}^n_{\mathfrak{n}} \xi^n \to \xi \text{ in } \mathcal{C}^{\alpha
    {-} 2}_{\mathfrak{n}}([0,L]^{d}), 
  \end{equation}
  as well as:
  \begin{equation}\label{eqn:support_proof_lemma_eq2} 
    \sup_n \mathbb{E} [ \| n^{{-} d/2} (\xi^n)_+ \|_{ \mathcal{C}^{{-}
    \varepsilon}_{\mathfrak{n}}(\Lambda_{n})} + \| n^{{-} d/2} (\xi^n)_+
    \|_{L^2(\Lambda_n)}] < {+} \infty,
  \end{equation}
  with \(\mathcal{E}^n_{\mathfrak{n}} n^{{-} d/2} (\xi^n)_+ \to \nu\) in 
  \(\mathcal{C}^{ {-} \varepsilon}_{\mathfrak{n}}([0,L]^d )\). 
  Moreover, in dimension \(d = 2\), we have (recall \(\kappa_{n}\) from
  \eqref{eqn:renormalization_constant}): 
  \begin{equation}\label{eqn:support_proof_lemma_eq3} 
    \sup_{n} \mathbb{E} [ \| X^n_{\mathfrak{n}}
    \|_{\mathcal{C}^{\alpha}_{\mathfrak{n}}(\Lambda_{n})} + \|
    (X^n_{\mathfrak{n}} \reso \xi^n) {-} \kappa_{n} \|_{\mathcal{C}^{2 \alpha
    {-} 2}_{\mathfrak{n}}(\Lambda_{n})} ] < {+} \infty 
  \end{equation}
  as well as \(\mathcal{E}^n_{\mathfrak{n}}
  X^n_{\mathfrak{n}} \to X_{\mathfrak{n}} \text{ in }
  \mathcal{C}^{\alpha}_{\mathfrak{n}}([0,L]^d),\)
  and \(\mathcal{E}^n_{\mathfrak{n}} (X^n_{\mathfrak{n}} \reso \xi^n {-}
  \kappa_{n}) \to X_{\mathfrak{n}} \diamond \xi \text{ in } \mathcal{C}^{2
  \alpha {-} 2}_{\mathfrak{n}}([0,L]^{d})\). Once these bounds and convergences
  are established, and in view of \cite[Lemma 2.4]{PerkowskiRosati2019}, the
  Lemma follows from Skorohod's representation theorem. So far we have proven
  convergence a.s. for fixed \(L, \alpha, \ve\). The extension to all
  \(L, \alpha, \ve\) follows as in Corollary \ref{cor:persistence_coupling_all_L}.
  To find convergence for all \(\omega\) we set all functions to zero on a
  null-set.

  \textit{Step 1.} We now observe that the bound and convergence from
  \eqref{eqn:support_proof_lemma_eq1} as well as the bound and convergence for
  \(X^{n}_{\mf{n}}\) from \eqref{eqn:support_proof_lemma_eq3} are similar to and
  simpler than the bound for \(X^{n}_{\mf{n}} \reso \xi^{n}\). Also, Equation
  \eqref{eqn:support_proof_lemma_eq2} and the following convergences are
  analogous to \cite[Appendix B]{PerkowskiRosati2019}. We are left with proving
  the bound and convergence of \(X^{n}_{\mf{n}} \reso
  \xi^{n}\) from \eqref{eqn:support_proof_lemma_eq3}.
  
  \textit{Step 2.} First, we establish the uniform bounds. We will
  derive only bounds in spaces of the kind \(B^{\mathfrak{n},
  \beta}_{p,p} (\Lambda_{n})\) for appropriate \(\beta\) and any \(p\)
  sufficiently large.  The results on the H\"older scale then follow by Besov
  embedding. In order to avoid confusion, we will omit the subindex
  \(\mathfrak{n}\) in the noise terms. 
  We write sums as discrete integrals against scaled measures with the following
  definitions:
  \[ 
    \int_{\Theta_{n}} \ud x \ f(x) = \sum_{x \in \Theta_n} \frac{f(x)}{n^d}, \ \
    \int_{\Xi_{n}} \ud k \ f(k) = \sum_{k \in \Xi_{n}} \frac{f(k)}{N^{d}} , \ \
    \int_{ \{{-} 1, 1\}^d} \ud \mathfrak{q} \ f(\mathfrak{q}) =
    \sum_{\mathfrak{q} \in \{{-} 1, 1\}^{d}} f(\mathfrak{q}). 
  \]
  For \(k_{1}, k_{2} \in \Xi_{n}\) and \(\mf{q}_{1}, \mf{q}_{2} \in \{ {-} 1,
  1\}^{d}\) we moreover adopt the notation: \(k_{[12]} = k_{1} {+} k_{2},
  \mf{q}_{[12]} = \mf{q}_{1} {+} \mf{q}_{2}\) and \( ( \mf{q} \circ k)_{[12]} =
  \mf{q}_{1} \circ k_{1} {+} \mf{q}_{2} \circ k_{2}\). 
  We first compute:
  \begin{align*}
     \Delta_{j}  \Pi_e (\xi^{n} & \reso X^{n}) (x) 
     = \int\limits_{ (\{{-} 1, 1\}^{d} \times \Xi^{+}_{n})^{2}} \ud
     \mathfrak{q}_{12} \ud k_{12}  \ N^{d} \nu_{k_1}
     \nu_{k_2} e^{2 \pi \iota \langle x, (\mathfrak{q} \circ k )_{[12]}
     \rangle} \cdot \\
     & \qquad \qquad \qquad \qquad \qquad \cdot \varrho_{j} ( (\mathfrak{q}
     \circ k )_{[12]}) \psi_{0}(k_{1}, k_{2})
     \frac{ \chi(k_{2})}{l^{n}(k_{2})} \langle \xi^{n}, \mathfrak{n}_{k_1}
     \rangle \langle \xi^{n}, \mathfrak{n}_{k_{2}} \rangle \\
     & \qquad \qquad \ \ = \int\limits_{ (\{{-} 1, 1\}^{d} \times
     \Xi^{+}_{n})^{2}}  \ud  \mathfrak{q}_{12} \ud k_{12} \
     1_{\{k_1 \neq k_2\}} N^{d} \nu_{k_1} \nu_{k_{2}}e^{2 \pi \iota \langle x,
     ( \mathfrak{q} \circ k )_{[12]}  \rangle} \cdot \\
     & \qquad \qquad \qquad \qquad \qquad \cdot \varrho_{j} ( (\mathfrak{q} \circ k)_{[12]}) \psi_{0}(k_{1}, k_{2})
     \frac{ \chi(k_{2})}{l^{n}(k_{2})}  \langle \xi^{n},
     \mathfrak{n}_{k_1} \rangle \langle \xi^{n}, \mathfrak{n}_{k_{2}} \rangle + \operatorname{Diag}
  \end{align*}
  where  \(\operatorname{Diag}\) indicates the integral over the set \( \{k_{1}
  = k_{2}\}\). The first term can be bounded by a generalized discrete BDG
  inequality for multiple discrete stochastic integrals, see \cite[Proposition
  4.3]{ChoukGairingPerkowski2017}. We can thus bound for arbitrary \(\l \in
  \NN\):
   \begin{align*}
	\mathbb{E}& [ |\Delta_{j} (\Pi_e(\xi^{n} \reso X^{n}) (x) {-}
	\kappa_n)|^{p}] \\ & \lesssim \left[ \int \ud \mathfrak{q}_{12}
	\ud k_{12} \ \bigg\vert \varrho_{j} ( (\mathfrak{q} \circ
	k )_{[12]}) \psi_0(k_1, k_2) \frac{ \chi(k_{2})}{l^{n}(k_{2})}
	\bigg\vert^{2} \right]^{\frac{p}{2} } \mathbb{E} [\langle \xi^{n},
	\mathfrak{n}_{\ell} \rangle^{p}]^{2} + \mathbb{E} [
	|\operatorname{Diag} {-} 1_{ \{j = {-} 1\}} \kappa_n|^{p} ].  
  \end{align*}
  For the first term on the right hand side we have:
  \begin{align*}
    &\int\limits_{(\{{-}  1, 1\}^{d} \times \Xi_{n}^+)^{2}} \!\!\!\!\!\!\!\!\!\! \ud
    \mathfrak{q}_{12} \ud k_{12} \ \bigg\vert \varrho_{j} ( (\mathfrak{q} \circ
    k)_{[12]}) \psi_0(k_1, k_2) \frac{ \chi(k_{2})}{l^{n}(k_{2})}
    \bigg\vert^{2} 
    = \int_{\Xi_{n}^{2}} \!\!\! \ud k_{12} \ \bigg\vert
    \varrho_{j} (k_{[12]}) \psi_0(k_1, k_2) \frac{ \chi(k_{2})}{l^{n}(k_{2})}
    \bigg\vert^{2}  \\ 
    &\hspace{30pt} \lesssim \sum_{i\ge j-\ell}
    \int_{\Xi_{n}^{2}} \ud  k_{12}  \ 1_{ \{|k_1+k_2| \sim 2^{j}\}} 1_{ \{|k_2|
    \sim 2^{i}\}}2^{{-} 4i} \lesssim \sum_{i \ge j-\ell} 2^{jd} 2^{i(d {-} 4)}
    \lesssim 2^{2j(d-2)}, 
  \end{align*}
  which is of the required order (and we used that $d<4$). Let us pass to the
  diagonal term. We first smuggle in the expectation of $\operatorname{Diag}$:
  \begin{align*}
    \mathbb{E} [ |\operatorname{Diag} {-} \mathbb{E} [\operatorname{Diag}
    ]|^{p} ] = \mathbb{E} \left[ \bigg\vert \int_{\Xi_{n}^{+} \times (\{{-} 1,
    1\}^{d})^{2}} \ud \mathfrak{q}_{12} \ud k \ \nu_{k}^{2} e^{2 \pi \iota \langle
    x, \mathfrak{q}_{[12]} \circ k  \rangle}
    \varrho_{j}   (\mathfrak{q}_{[12]} \circ k ) \frac{
    \chi(k)}{l^{n}(k)} \eta(k) \bigg\vert^{p} \right], 
  \end{align*}
  where we have lost the factor \(N^{d}\) due to the normalization of the 
  integral in \(k\) and \(\eta(k) = \langle \xi^{n}, \mathfrak{n}_{k}
  \rangle^{2} {-} \mathbb{E} [ \langle \xi^{n} , \mathfrak{n}_{k} \rangle^{2}]
  = \langle \xi^{n}, \mathfrak{n}_{k} \rangle^{2} {-}1 \) is sequence of
  centered i.i.d random variables. Therefore, we can use the same martingale
  argument as above to bound the integral by:
  \begin{align*}
    \mathbb{E} [ |\operatorname{Diag} {-} \mathbb{E} [\operatorname{Diag} ]|^{p}
    ] & \lesssim  \bigg( \int_{\Xi_{n}^{+}} \ud k \bigg\vert \int_{( \{{-} 1,
    1\}^{d})^{2}} \ud \mathfrak{q}_{12} \ \varrho_{j} (\mathfrak{q}_{[12]}
    \circ k) \bigg\vert^2  \bigg\vert \frac{ \chi(k)}{l^{n}(k)} \bigg\vert^{2}
    \EE[|\eta(k)|^p]^{\frac{2}{p}}\bigg)^{\frac{p}{2} }  \\ & \lesssim
    \bigg(\int_{x\in \RR^d: |x| \gtrsim 2^j } \frac{1}{|x|^4} \ud x\bigg)^{p/2}
    \lesssim 2^{j(d-4)} = 2^{j (\frac{d}{2} - 2)} 
  \end{align*}
  whenever $d < 4$, which is even better than the  bound for the off-diagonal
  terms.  We are hence left with a last, deterministic term:
  \[ 
    \int_{\Xi_{n}^{+} \times (\{{-} 1, 1\}^{d})^{2}} \ud \mathfrak{q}_{12} \ud k
    \ \nu_{k}^{2} e^{2 \pi \iota \langle x, \mathfrak{q}_{[12]} \circ k \rangle}
    \varrho_{j} (\mathfrak{q}_{[12]} \circ k) \frac{ \chi(k)}{l^{n}(k)}  \ \ - \ \
    1_{\{j = {-} 1\}} \kappa_n.  
  \]
  We split up this sum in different terms according to the relative value of
  \(\mathfrak{q}_{1}, \mathfrak{q}_{2}\). If \(\mathfrak{q}_{1} = {-}
  \mathfrak{q}_{2}\) (there are $2^d$ such terms) the sum does not depend on
  \(x\) and it disappears for \(j \geq 0\). Let us assume \(j = {-} 1\). We
  are then left with the constant:
  \begin{align*}
    2^{d} \int_{\Xi_{n}^{+}} \ud k \ \nu_{k}^{2} \frac{ \chi(k)}{l^{n}(k)} -
    \kappa_{n} =  \int_{\Xi_{n}} \ud k \ \frac{ \chi(k)}{l^{n}(k)} - \kappa_{n}.
  \end{align*}
  Note that the sum on the left-hand side diverges logarithmically in \(n\) and
  we now show how to renormalize with \(\kappa_n\). To clarify our computation
  let us also introduce an auxiliary constant \(\bar{\kappa}_n = \int_{\Xi_n}
  \ud k \  \overline{\nu}_{k}^{2} \frac{\chi(k)}{l^n(k)}, \) where \(
  \overline{\nu}_{k} = 2^{{-} \# \{ i \colon k_i = \pm n\}/2}\).  For \(x \in
  \RR^{d}, r \geq 0\), let us indicate with \(Q_{r}^{n}(x) \subseteq
  \TT_{n}^{d}\) the box \( Q^{n}_{r} (x) = \{ y \in \tnd \colon |y {-}
  x|_{\infty} \leq r/2 \}\) ( \(| \cdot|_{\infty}\) being the maximum of the
  component-wise distances in \(\TT^{d}_{n}\)). Then note that we can bound
  uniformly over \(n\) and \(N\):
  \begin{align*}
    \vert \kappa_{n} {-}   \bar{\kappa}_n \vert & = \bigg\vert \int_{\TT^d_n}
    \ud k \ \frac{\chi(k)}{l^n(k)} {-} \int_{\Xi_{n}} \ud k \
    \overline{\nu}_{k}^{2} \frac{\chi(k)}{l^n(k)} \bigg\vert = \bigg\vert
    \sum_{k \in \Xi_n} \int_{Q_{ \frac{1}{N} }^{n}(k)} \ud k^{\prime} \
    \frac{\chi(k {+} k^\prime)}{ l^n(k{+}k^\prime)} {-}
    \frac{\chi(k)}{l^n(k)}\bigg\vert \\ 
    & \lesssim \frac{1}{N} \bigg( 1 + \frac{1}{N^{d}} \sum_{k \in \Xi_n} \sup_{
    \vartheta \in Q_{ \frac{1}{N} } (k)}\frac{ \chi(k) }{ ( l^n(\vartheta) )^2}
    |\nabla l^n (\vartheta) | \bigg) \lesssim \frac{1}{N} \bigg( 1 {+}
    \frac{1}{N^{d}} \sum_{k \in \frac{1}{N} \ZZ^d} \frac{\chi(k)}{|k|^3} \bigg)
    \lesssim \frac{1}{N}, 
  \end{align*}
  where we have used that $d=2$, $|l^n(\vartheta)| \gtrsim |\vartheta|^2$ on
  $[{-}n/2, n/2]^d$ as well as $|\nabla  l^n(\vartheta)| \lesssim |\theta|$ on
  $[{-}n/2, n/2]^d$. Similar calculations show that the difference converges:
  $\lim_{n \to \infty}\kappa_n {-} \bar{\kappa}_n \ \in \RR.$ We are now able to
  estimate:
  \begin{align*} 
    \bigg\vert \int_{\Xi_{n}} \ud k \ \frac{ \chi(k)}{l^{n}(k)} - \kappa_n
    \bigg\vert \lesssim 1 + |\bar{\kappa}_n {-} \kappa_n| \lesssim 1 
  \end{align*}
  where we used that the sum on the boundary $\partial \Xi_n$ converges to zero
  and is thus uniformly bounded in $n$.  For the same reason, the above
  difference converges to the limit \( \lim_{ n \to \infty}
  \overline{\kappa}_{n} {-} \kappa_{n} \in \RR\).
 
  For all other possibilities of $\mathfrak{q}_1, \mathfrak{q}_2$ we will show
  boundedness in a distributional sense.  If \(\mathfrak{q}_{ 1} = \mathfrak{
  q}_{2}\) we have: 
  \begin{align*} 
    \bigg\vert \int_{\Xi_{n}^{+}} \ud k \ \nu_{k}^{2} e^{2 \pi \iota \langle x,
    2 \mathfrak{q}_{1} k \rangle} \varrho_{j}(2 k) \frac{\chi(k)}{l^{n}(k)}
    \bigg\vert \lesssim 2^{j(d {-} 2)}.
  \end{align*} 
  Finally, if only one of the two components of \(\mathfrak{q}_1,
  \mathfrak{q}_2\) differs (let us suppose it is the first one) we find ( with \(x
  = (x_1,x_2)\) and \(k = (k_1,k_2)\)):
  \begin{align*} 
    \bigg\vert \int_{\Xi^+_{n}} \ud k \ \nu_{k}^{2} e^{2 \pi \iota 2x_2 k_2 }
    \varrho_{j}(2k_2) \frac{ \chi(k)}{l^n (k)}  \bigg\vert \lesssim \left(
    \sum_{ k_1 \geq 1}  \frac{1}{|k_1|^{2 \theta}} \right) \left( \sum_{k_2 \geq
    1} \frac{\varrho_{j}(2k_2)}{|k_2|^{2 (1 {-} \theta)}}  \right) \lesssim
    2^{j \varepsilon} 
  \end{align*} 
  for any \(\varepsilon > 0\), up to choosing \(\theta \in (1/2,1)\)
  sufficiently close to \(1/2\).

  \textit{Step 3.} Now we briefly address the convergence in distribution.
  Clearly the previous calculations and compact embeddings of H\"older-Besov
  spaces guarantee tightness of the sequence 
  $X^n_{\mathfrak{n}} \reso \xi^n {-} \kappa_n$ in the required H\"older spaces for
  any $\alpha < 2{-}d/2$. We have to uniquely identify the distribution of any
  limit point. Whereas for $\xi, X^n_{\mathfrak{n}}$ the limit points are Gaussian and
  uniquely identified as white noise $\xi$ and $\Delta_{\mathfrak{n}}^{-1}
  \chi(D) \xi$ respectively, the resonant product requires more care, but we
  can use the same arguments as in \cite[Section~5.1]{MartinPerkowski2017} for
  higher order Gaussian chaoses.  
\end{proof}

\section{Killed rSBM}\label{sect:killed_rsbm}

In this last section we briefly introduce a killed version of the rSBM described
in \cite{PerkowskiRosati2019}. This process arises as a scaling limit of a
branching random walk in a random environment in which a walker is killed once
he leaves a box of size $L \in 2\mathbb{N}$.  Recall that we consider the
lattice approximation \(\Lambda_{n}^L = \{x \in \znd \ \colon \ x \in [{-} L/2,
L/2]^d \}\) (we explicitly write the dependence on \(L\) because we will let
\(L\) vary). Define in addition the space of functions $E^L = \big\{ \eta \in
\NN^{\Lambda^L_n}_0 \ : \ \eta(x) = 0 , \forall x \in \partial \Lambda^L_n
\big\}$. Recall that the last point of Lemma \ref{lem:renormalisation-neumann}
allows us to apply the results of \cite{PerkowskiRosati2019}. We work in the
following framework. 

\begin{assumption}\label{assu:framework_krsbm}

  Let \(\xi^{n}\) be the sequence of random variables on \(\Omega\) constructed
  in Lemma \ref{lem:renormalisation-neumann} and write:
  \[
    \xi^{n}_{e}( \omega, x) = \xi^{n}(\omega, x) {-} c_{n}( \omega)1_{
    \{ d = 2\}}. 
  \]
  Fix \(\varrho= d/2\), let \(u^{n}( \omega, t, x)\) be the process
  constructed in \cite[Definition 2.6]{PerkowskiRosati2019} and let \(\mu^{n}(\omega, t)\) be
  the measure associated to it. Such process lives on a probability space:
  \[ \big( \Omega \times \overline{\Omega}, \mF, \PP \ltimes \PP^{\omega, n}\big),\]
  where \(\PP^{\omega}\) is the quenched law of \(u^{n}\), conditional on the
  environment \(\xi^{n}(\omega)\), for \(\omega \in \Omega\).

\end{assumption}

The BRWRE \(u^{n}\) does not keep track of the individual particles (all
  particles are identical, only their position matters, cf \cite[Appendix
A]{PerkowskiRosati2019}). We shall also consider the labelled process, which
distinguishes individual particles and kill all particles which leave a given
box. We thus introduce the space
\(E_{\operatorname{lab}} = \bigsqcup_{m \in \NN} ( \frac{1}{n} \ZZ^d \cup
\{\Delta\})^m\),
where $\bigsqcup$ denotes the disjoint union, endowed with the discrete
topology. Here $\Delta$ is a cemetery state. For $\eta \in
E_{\operatorname{lab}}^{n} $ we write $ \operatorname{ dim } (\eta) = m$ if $\eta
\in ( \frac{1}{n} \ZZ^d \cup \{ \Delta \})^m$. A rigorous construction of the
process below follows a in \cite[Appendix A]{PerkowskiRosati2019}.

\begin{definition}\label{def:labelled_markov_process_definition}
  Fix $\omega \in \Omega$ and \(X^{n}_{0} \in E_{\mathrm{lab}}^{n}\) with
  \(\mathrm{dim}(X^{n}_{0}) = \lfloor n^{\varrho} \rfloor, (X^{n}_{0} )_{i} = 0,
  i = 1 \dots \lfloor n^{\vr} \rfloor \). Construct the Markov jump process
  \(X^{n}(\omega)\) on \(E_{\operatorname{lab}}^{n}\) via \(X^{n}(0) = X^{n}_0\)
  with generator:
  \begin{align*} 
    \mL_{\operatorname{lab}}^{\omega} & ( F ) (\eta) = \sum_{i = 1}^{
    \operatorname{ dim } (\eta)} 1_{ \{ \frac{1}{n} \ZZ^d\}}(\eta_i) \Big[
      \ \ \sum_{|y {-}  \eta_i| = n^{{-} 1}}\big( F( \eta^{i \mapsto y}) {-}
      F(\eta)\big) \\ 
    & + (\xi^n)_+ (\omega, \eta_i) \big( F(\eta^{i, +}) {-} F(\eta)\big) {+}
    (\xi^n)_- (\omega, \eta_i) \big( F(\eta^{i, -}) {-} F(\eta)\big) \Big],
  \end{align*}
  where \(\eta^{i \mapsto y}_j= \eta_j(1 {-} 1_{\{i\} }(j)) {+} y 1_{\{i\}}(j)
  \) and \(\eta^{i, +}_j = \eta_j 1_{ [0, \operatorname{ dim } (\eta)]}(j) {+}
  \eta_i 1_{\{ \operatorname{dim} (\eta) {+} 1\}}(j)\) as well as \(\eta^{i,
  -}_j = \eta_{j}(1 {-} 1_{\{i\}}(j)) {+} \Delta 1_{\{i\}}(j)\), on the domain
  \( \mD( \mL^{\omega}_{\mathrm{lab}})\) of functions \(F\)
  is such that the right hand-side is bounded. We can then redefine the process
  \[
    u^{n}(\omega, t,x)= \# \{i \in \{i, \ldots,  \operatorname{ dim }
    (X^{n}(\omega, t))\} \ \colon \ X_i( \omega,t) =x\} 
  \] 
  which has the same quenched law \(\PP^{\omega}\) as the process above.

  Similarly, for \(i \in \NN\) consider \(\tau_i^{n,L}(\omega) = \inf \{ t \geq
  0 \colon \mathrm{dim}(X^{n}(\omega, t)) \geq i \text{ and } X^{n}_i(t) \in
  \partial \Lambda^{L}_{n}\}\). Define  $X^{n,L}(\omega, t) \in
  E_{\mathrm{lab}}^{n}$ by \( \mathrm{dim}(X^{n, L }(\omega, t)) =
  \mathrm{dim}(X^{n}(\omega, t))\) and \(X^{n, L}_{i}(\omega , t) =
  X^{n}_{i} (\omega, t) 1_{ \{t < \tau^{n , L}_{i}( \omega) \}} + \Delta
  1_{ \{ \tau^{n, L}_{i}(\omega) \leq t\}}\).
  
  Define $u^{n, L}$ taking values in $E^L$ by
  \[ 
    u^{n, L}(\omega, t, x) = \# \{ i \in \{ 1, \ldots , \mathrm{dim} ( X^{n,
      L}(\omega, t)) \} \ : \ X^{n, L}_{i}(\omega, t) = x
    \}.
  \]
  \end{definition}

Write $\mM(({-}L/2, L/2)^d)$ for the set of all finite positive measures on
$({-}L/2, L/2)^d$ and for \(\mu, \nu\) in this space we say \(\mu \geq \nu\) if
also \(\mu {-} \nu\) is a \emph{positive} measure. The following result is now
easy to verify (cf. \cite[Appendix A]{PerkowskiRosati2019}).

\begin{lemma} 

  For any $\omega \in \Omega$ the process $ t \mapsto u^{n,L}( \omega, t,
  \cdot)$ is a Markov process with paths in $\DD([0, {+}
  \infty); E^L)$, associated to the generator $\mL^{n , \omega}_L\colon C_b(E^L)
  \to C_b(E^L)$ defined via:
  \begin{align*} \mL^{n , \omega}_L(F) (\eta) & =  \sum_{x
    \in \Lambda^L_n \setminus \partial \Lambda^L_n} \eta_x \cdot  \bigg[  \sum_{x
      \sim y} n^2( F(\eta^{x \mapsto y}) {-} F(\eta)) \\
      &\hspace{80pt} {+} (\xi^n_e)_+(\omega, x) [ F( \eta^{x{+}}) {-} F(\eta)] + (\xi^n_e)_{-}(\omega, x)  [F(\eta^{x{-}}) {-} F(\eta)]\bigg], 
  \end{align*}
  where for $\eta \in E^L$ we define $\eta^{x \mapsto y}(z) = (\eta(z) {-} 1_{\{ z
  = x\}} {+} 1_{\{z = y, \ y \not\in \partial\Lambda^L_n\}})_+$ and $\eta^{x
  \pm}(z) = (\eta(z) \pm 1_{ \{ z = x\}})_{+}$.  
  We associate to \(u^{n, L}(\omega, t)\) a measure:
  \begin{equation}\label{eqn:defn_of_mu_n_dirichlet}
    \mu^{n, L}(\omega, t) (\varphi) = \sum_{x \in \Lambda^L_n} \lfloor n^{-\vr}
    \rfloor u^{n , L}(\omega, t, x) \varphi (x), \qquad \forall \varphi \in
    C(({-} L/2, L/2)^{d}).
  \end{equation}
  Finally:
  \begin{equation}\label{eqn:comparison_discrete_measures}
    \mu^{n, L} (\omega, t) \leq \mu^{n, L {+} 2} (\omega, t) \leq \dots \leq
    \mu^{n}(\omega, t) \qquad \forall \omega \in \Omega, t \geq 0. 
  \end{equation}
\end{lemma}

When studying the convergence of the process \(\mu^{n ,L}\), special care has to
be taken with regard to what happens on the boundary of the box. Indeed a
function \(\varphi \in C^{\infty}([{-} L/2, L/2]^d )\) (i.e. smooth in the
interior with all derivatives continuous on the entire box) is not smooth in the
scale of spaces \(B^{ \mf{l}, \alpha}_{p, q}\) for \(\mf{l} \in \{ \mf{d},
\mf{n}\}\), since it does not satisfy the required boundary conditions: a priori
it only lies in the above space for \(\alpha=0\) and any value of \(p, q\).  For
this reason we consider a weaker kind of convergence for the processes \(\mu^{n,
L}\) than one might expect. We write
\[
	\mM_0^L = \big( \mM( ({-} L/2, L/2)^d ), \tau_v\big)
\]
of finite positive measures on \( ({-} L/2, L/2)^d\) endowed with the vague
topology \(\tau_v\) (cf. \cite[Section 3]{DawsonPerkins2012}), i.e. \(\mu^n \to
\mu\) in \(\mM_0^L\) if \( \mu^n(\varphi) \to \mu( \varphi),\) for all \(\varphi
\in X\), where \(X\) can be chosen to be either the space \(C^{\infty}_c( ({-}
L/2, L/2)^d)\) of smooth functions with compact support or the space \(C_0( ({-}
L/2, L/2)^d)\) of continuous functions which vanish on the boundary of the box
(the latter is a Banach space, when endowed with the uniform norm). The reason
why this topology is convenient is that sets of the form \(K_R \subset
\mM_0^L\), with \( K_R = \{ \mu \in \mM^L_0 \ : \ \mu(1) \leq R\}\) are compact.
In this setting it is also important to remark the following embedding, which
follows from a short calculation.

\begin{remark}\label{rem:continuous_embedding_dirichlet}
  
  For $\alpha > 0$ there is a continuous (in the sense of Banach spaces) embedding
	  \[
		  \mC^{\alpha}_{\mf{d}}([{-} L/2, L/2]^d) \hookrightarrow C_0(
		({-} L/2, L/2)^d).
	 \]

      \end{remark}

Now we can pass to study the convergence of the killed process.

\begin{lemma}\label{lem:compact_containment_Dirichlet}

  We can bound the mass of the killed process locally uniformly in time. Namely,
  for any \(\omega \in \Omega\):
  \[ 
    \lim_{R \to \infty} \sup_n \PP^{\omega, n} \Big( \sup_{t \in [0, T]}\mu^{n,
    L}(\omega, t)(1) \geq R \Big) = 0 , \qquad \sup_n \sup_{t \in [0, T]} \| T^{n,
    \mf{d}, L , \omega }_{t} 1  \|_{\infty} < {+} \infty.
  \]

\end{lemma}

\begin{proof}
  The first bound follows from comparison with the process on the whole real
  line (i.e. Equation \eqref{eqn:comparison_discrete_measures}), see
  \cite[Corollary 4.3]{PerkowskiRosati2019}. The second bound follows from Theorem
  \ref{thm:convergence_PAM_Dirichlet} because the antisymmetric extension of $1$
  is in $L^\infty$: we have \(| \Pi_{o} 1 (\cdot)| \equiv 1\).
\end{proof}

\begin{lemma}

  For every \(\omega \in \Omega\) the sequence \(\{ t \mapsto \mu^{n, L}(\omega,
  t)\}_{n \in \NN}\) is tight in the space \(\DD(\RR_{\ge 0}; \mM^L_0)\). Any
  limit point \(\mu^L(\omega)\) lies in \(C(\RR_{\ge 0}; \mM_0^L)\).

\end{lemma}

\begin{proof}
   
  We want to apply Jakubowski's tightness criterion \cite[Theorem
  3.6.4]{DawsonPerkins2012}. The sequence \(\mu^{n, L}\) satisfies the compact
  containment condition in view of Lemma
  \ref{lem:compact_containment_Dirichlet}. The tightness thus follows if we
  prove that the sequence \(\{ t \mapsto \mu^n (t)(\varphi)\}_{n \in \NN}\) is
  tight in \(\DD([0,T]; \RR)\) for any \(\varphi \in C^{\infty}_{c}( ({-} L/2,
  L/2)^d)\). Here we can follow the calculation of \cite[Lemma
  4.2]{PerkowskiRosati2019} (only simpler, since we do not need weights), using
  the results from Theorem \ref{thm:convergence_PAM_Dirichlet}. The continuity
  of the limit points is shown as in \cite[Lemma 4.4]{PerkowskiRosati2019}.

\end{proof}
 
We will characterize the limit points of $\{\mu^{n,L}\}_{n \in \NN}$ in a
similar way as the rough super-Brownian motion, and for that purpose we need to
solve the following equation (for any \(\omega \in \Omega, L \in 2 \NN\)):
\begin{equation}\label{eqn:PDE_laplace_duality_continuous_dirichlet}
  \partial_t\varphi = \mH_{\mfd, L}^{\omega} \varphi {-} \nu \varphi^2, \ \
  \varphi(0) = \varphi_0, \ \ \varphi(t, x) = 0, \  \forall (t, x) \in
  (0,T]\times\partial [{-}L/2,L/2]^d, 
\end{equation}
where we define \(\varphi\) a solution to
\eqref{eqn:PDE_laplace_duality_continuous_dirichlet} if
\[ 
  \varphi(t) = T^{\mf{d}, L, \omega}_{t} \varphi_{0} - \nu \int_{0}^{t}
  T^{\mf{d}, L, \omega}_{t {-} s} [\varphi^{2}(s)]  \ud s. 
\] 

\begin{lemma}\label{lem:existence_solutions_duality_PDE_dirichlet} 
  
  Fix \(\omega \in \Omega, L \in 2\NN\). For $T>0$ and $\varphi_0 \in
  C^{\infty}_c(({-}L/2, L/2)^d)$ with $\varphi_0 \ge 0$ and $\vartheta$ as in
  Theorem \ref{thm:convergence_PAM_Dirichlet}, there exists a unique
  (paracontrolled in $d=2$) solution $\varphi \in
  \mL^{\vartheta}_{\mfd}([{-}L/2, L/2]^d)$ to
  \eqref{eqn:PDE_laplace_duality_continuous_dirichlet} and the following bounds
  hold: 
  \[ 0 \le \varphi(t) \le T^{\mfd, L , \omega}_t \varphi_0, \qquad \|
    \varphi \|_{\mL^{\theta}_{\mfd}([ {-} L/2, L/2]^{d})} \lesssim e^{C\| \{T^{
  \mfd, L , \omega}_t \varphi_0\}_{t \in [0,T]} \|_{CL^{\infty}([ {-} L/2,
  L/2]^{d})}}.  \]

\end{lemma}

The proof is analogous to the one of \cite[Proposition 4.5]{PerkowskiRosati2019}. 
We thus arrive at the following
description of the limit points of $\{ \mu^{n,L}\}_{n \in \NN}$.

\begin{theorem}\label{thm:convergence_defn_killed_rsbm} 
  
  For any \(\omega \in \Omega\) and \(L \in 2 \NN\), under Assumption
  \ref{assu:framework_krsbm}, there exists $\mu^L(\omega) \in C(\RR_{\ge 0};
  \mM^L_0)$ such that $\mu^{n,L}(\omega) \to \mu^L(\omega)$ in distribution in
  $\DD(\RR_{\ge 0}; \mM^L_0)$. The process $\mu^L(\omega)$ is the unique (in
  law) process in $C(\RR_{\ge 0}; \mM^L_0)$ which satisfies one (and then all)
  of the following equivalent properties with $ \mF^{\omega} = \{ \mF_t^{\omega}
  \}_{t \ge 0}$ being the usual augmentation of the filtration generated by
  $\mu^L(\omega)$.
  
  \begin{enumerate}[label=(\roman*)] 
    \item For any $t \ge 0$ and $\varphi_0 \in C_c^{\infty}(({-}L/2, L/2)^d),
      \varphi_0 \ge 0$ and for \(U^{\mf{d}, L, \omega}_t \varphi_0\) the
      solution to Equation \eqref{eqn:PDE_laplace_duality_continuous_dirichlet}
      with initial condition $\varphi_0$ the process
      \[ N^{\varphi_0}_t(s) = e^{-\langle \mu^{L}(\omega, s), U^{\mf{d}, L ,
      \omega}_{t {-} s} \varphi_0 \rangle }, \qquad s \in [0, t] \]
      is a bounded continuous $\mF^{\omega}-$martingale.
		
    \item  
      For any $\varphi \in \mD_{\mH_{\mfd, L}^{\omega}}$ the process: 
      \[ 
	K^{\varphi}(t) = \langle \mu^L(\omega, t), \varphi \rangle{-} \langle
	\delta_0, \varphi \rangle {-}\int_0^t \ud r \ \langle \mu^L(\omega, r),
	\mH_{\mfd, L}^{\omega} \varphi \rangle, \qquad t \in [0, T] 
      \] 
      is a continuous $\mF^{\omega}-$martingale, square-integrable on $[0,T]$
      for all $T > 0$, with quadratic variation 
      \[
	\langle K^{\varphi} \rangle_t =2\nu  \int_0^t \ud r \ \langle
	\mu^L(\omega, r), \varphi^2 \rangle.  
      \] 
  \end{enumerate}
      
\end{theorem}

\begin{proof}
  The proof is almost identical to the one of \cite[Theorem
  2.13]{PerkowskiRosati2019}. The main difference is that here we only test
  against functions with zero boundary conditions and thus use the results from
  Section \ref{sect:PAM_dirichlet_bc}.
\end{proof}

We call the above process the killed rSBM on $(-\frac{L}{2},\frac{L}{2})^d$.
Note that we can interpret the killed rSBM as an element of $C(\RR_{\ge 0};
\mM(\RR^d))$ by extending it by zero, i.e.  \(\mu^L(\omega, t, A) = \mu^L(
\omega, t, A \cap ({-} L/2, L/2)^d)\) for any measurable \(A \subset
\RR^d\). This allows us to couple infinitely many killed rSBMs with a rSBM
on $\RR^d$ so that they are ordered in the natural way.

\begin{corollary}\label{cor:persistence_coupling_all_L}
  For any \(\omega \in \Omega,\) under Assumption \ref{assu:framework_krsbm},
  there exists a process 
  \[(\mu(\omega, \cdot), \mu^2(\omega, \cdot),
  \mu^4(\omega, \cdot), \dots)\] taking values in
  $C(\RR_{\ge 0}; \mM(\RR^d))^\NN$ (equipped with the product topology) such
  that $\mu$ is an rSBM and $\mu^L$ is a killed rSBM for all $L \in 2 \NN$ (all
  associated to the environment $\{\xi^n\}_{n \in \NN}$), and such that:
  \begin{equation}\label{eqn:ordered-killed-rsbm}
    \mu^2( \omega, t,A) \le \mu^4(\omega, t,A) \le \dots \le \mu(\omega, t,A)
  \end{equation}
  for all $t \ge 0$ and all Borel sets $A \subset \RR^d$.
\end{corollary}

\begin{proof}
  The construction~\eqref{eqn:defn_of_mu_n_dirichlet} of $\mu^n$ and $\mu^{n,L}$
  based on the labelled particle system gives us a coupling $(\mu^n, \mu^{n,2},
  \mu^{n,4}, \dots)$ such that for all \(\omega \in \Omega\)
  \[
     \mu^{n,2}(\omega, t,A) \le \mu^{n,4}(\omega, t,A) \le \dots \le
     \mu^n(\omega,t,A)
  \]
  for all $t \ge 0$ and all Borel sets $A \subset \RR^d$, where as above we
  extend $\mu^{n,L}$ to $\RR^d$ by setting it to zero outside of $(-\frac{L}{2},
  \frac{L}{2})^d$ (cf. Equation \eqref{eqn:comparison_discrete_measures}). By
  \cite[Theorem 2.13]{PerkowskiRosati2019} and
  Theorem~\ref{thm:convergence_defn_killed_rsbm} we get tightness of the
  finite-dimensional projections $(\mu^n, \mu^{n,2}, \dots, \mu^{n,L})$ for $L
  \in 2\NN$, and this gives us tightness of the whole sequence in the product
  topology. Moreover, for any subsequential limit $(\mu, \mu^2, \mu^4, \dots)$
  we know that $\mu$ is an rSBM and $\mu^L$ is a killed rSBM on $(-\frac{L}{2},
  \frac{L}{2})^d$.
  It is however a little subtle to obtain the
  ordering~\eqref{eqn:ordered-killed-rsbm}, because we only showed tightness in
  the vague topology on $\mM_0^L$ for the $\mu^{n,L}$ component. So we introduce
  suitable cut-off functions to show that the ordering is preserved along any
  (subsequential) limit: Let \(\chi^m \in C^{\infty}_{c}( ({-} L/2, L/2)^d)\),
  \(\chi^m \geq 0\) such that \(\chi^m =1\) on a sequence of compact sets
  \(K^m\) which increase to \( ({-} L/2, L/2)^d\) as $m \to \infty$. Note that
  on compact sets the sequence \(\mu^{n, L}\) converges weakly (and not just
  vaguely). We then estimate (in view of Equation
  \eqref{eqn:comparison_discrete_measures}) for $\varphi \in C_b(\RR^d)$ with
  $\varphi \ge 0$: 
  \begin{align*} 
    \langle \mu^L (t), \varphi \rangle =  \lim_{m \to \infty} \langle \mu^L (t),
    \varphi \cdot \chi^m \rangle & = \lim_{m \to \infty} \lim_{n \to \infty}
    \langle \mu^{n, L} (t), \varphi \cdot \chi^m \rangle \\ & \leq \lim_{m \to
    \infty} \langle \mu (t) , \varphi \cdot \chi^m \rangle = \langle \mu (t) ,
    \varphi \rangle, 
  \end{align*} 
  and similarly we get $\langle \mu^L (t), \varphi \rangle \le
  \langle \mu^{L'} (t), \varphi \rangle$ for $L \le L'$. Since a signed measure
  that has a positive integral against every positive continuous function must
  be positive, our claim follows.
\end{proof}

\bibliographystyle{plain}
\bibliography{rsbm}

\end{document}